\documentclass[reqno]{amsart}
\usepackage{amssymb}
\usepackage{amscd}
\usepackage{mathtools}

\numberwithin{equation}{section}

\theoremstyle{plain}

\newtheorem{theorem}[subsection]{Theorem}
\newtheorem{proposition}[subsection]{Proposition}
\newtheorem{lemma}[subsection]{Lemma}
\newtheorem{corollary}[subsection]{Corollary}

\theoremstyle{definition}

\newcommand{\f}{\mathbb{F}_q}
\newcommand{\ff}{\mathbb{F}_p}
\newcommand{\z}{\mathbb{Z}}
\newcommand\F{\mathbb{F}_q}

\newcommand\Z{\mathbb{Z}}

\DeclareMathOperator{\sign}{sign}

\usepackage[colorlinks=true,linkcolor=blue, citecolor=magenta]{hyperref}

\begin{document}

\title[Distinct coordinate solutions of linear equations over finite fields]{distinct coordinate solutions of linear equations over finite fields}

\author{Jiyou Li}
\address{Department of Mathematics, Shanghai Jiao Tong University, Shanghai, P.R. China}
\email{lijiyou@sjtu.edu.cn}

\author{Xiang Yu}
\address{Department of Mathematics, Shanghai Jiao Tong University, Shanghai, P.R. China}
\email{xangyu@alumni.sjtu.edu.cn}

\thanks{This work is supported by the National Science Foundation of China (11771280) and the National Science Foundation of Shanghai Municipal (17ZR1415400, 19ZR1424100). }


\begin{abstract}
Let $\f$ be the finite field of $q$ elements and $a_1,a_2,  \ldots, a_k, b\in \f$. We investigate $N_{\F}(a_1, a_2, \ldots,a_k;b)$,  the number of ordered solutions $(x_1, x_2, \ldots,x_k)\in\F^k$ of the linear equation
$$ a_1x_1+a_2x_2+\cdots+a_kx_k=b$$
with all $x_i$ distinct. We obtain an explicit formula for
 $N_{\F}(a_1,a_2, \ldots, a_k;b)$ involving combinatorial numbers depending on $a_i$'s.  In particular, we obtain closed formulas for two special cases. One is that $a_i, 1\leq i\leq k$ take at most three distinct values and the other is that $\sum_{i=1}^ka_i=0$ and $\sum_{i\in I}a_i\neq 0$ for any $I\subsetneq \{1,2,\dots,k\}$.

  The same technique works when $\f$ is replaced by $\z_n$,
  the ring of integers modulo $n$. In particular, we give a new proof for the main result given by Bibak, Kapron and Srinivasan
   (\cite{bibakgraph}), which generalizes a theorem of Sch\"{o}nemann via a graph theoretic method.
\end{abstract}

\maketitle

\section{Introduction}

Let $\F$ be the finite field of $q$ elements of characteristic $p$ and
$D$ be a subset in $\f$. Given $a_1, a_2, \ldots, a_k, b\in \F$, we are interested in the number of solutions of the linear equation over $\f$ 
\begin{equation}\label{eq:lineareq}
a_1x_1+a_2x_2+\cdots+a_kx_k=b,
\end{equation}
with  the restriction that all $x_i$ in $D$ are distinct, that is, the cardinality of the set
$$N_{D}(a_1, a_2, \dots, a_k; b)=\#\{(x_1, x_2, \ldots, x_k)\in D^k:a_1x_1+a_2x_2+\cdots+a_kx_k=b:x_i\neq x_j,\ \forall\ i\neq j\}.$$

This number is closely related to the reduced degree of a map over $\F$. Any map from $\F$ to $\F$ can be uniquely represented by a polynomial of degree at most $q-1$. The degree of such a polynomial is called the {\it reduced degree} of the map.  Suppose that the range of $f$ is $\{a_1, a_2, \ldots,a_q\}$(with multiplicity counted).  By the Lagrange interpolation formula, it is direct to check that $f$ is a polynomial of degree at most $q-2$ if and only if $\sum_{i=1}^q a_i=0$, and $f$ is a polynomial of degree at most $q-3$ if and only if $\sum_{i=1}^q a_i=0$ and $N_{\F}(a_1, a_2, \ldots,a_q;0)>0$.
In particular, it is well-known (see for example, \cite{Das}) that $N_{\F^*}(1, \omega, \omega^2, \dots, \omega^{q-2}; 0)$ counts the number of permutation polynomials
of $\text{degree}\leq q-3$ over $\f$, where $\omega$ is a primitive element of $\f$.  For more related work we refer to \cite{KP1, KP2, MW1}.

Furthermore,  $N_{D}(a_1, a_2, \dots, a_k; b)$ can be naturally regarded as a counting version of knapsack type problem over finite rings. In particular, when $a_i=1$, $1\leq i \leq k$, this is the counting version of subset sum problem, a well-known ${\bf \#P}$ problem in theoretical computer science.  It thus has many applications in coding theory and number theory.
For details we refer to \cite{cao,cheng2007deepholes,cheng2007distance,cheng2010complexity,gacs}.
Note that if all $a_i$'s lie in the prime field $\ff$, then this problem is a restricted composition problem over $\f$; see \cite{GW} for a broad generalization.

 Few results are known for arbitrary $a_i$'s,  even for special cases such as $D=\f$ or $D=\f^*$. In \cite{gacs}, G\'{a}cs et al.
proved that $N_{\F}(a_1, a_2, \dots,a_k;0)>0$ always holds except some obvious degenerate cases by using the polynomial method; later, Nagy \cite{Nagy} extended the result to cyclic groups.
Another result, proved by Li and Wan \cite{li2008}, gives a first
 explicit formula for $N_{D}(a_1, a_2, \dots,a_k;b) $ when $a_1=a_2=\cdots=a_k=1$ and $|D|\geq q-2$.
\begin{theorem}[\cite{li2008}, Theorem 1.2]\label{thm:li2008} %
	Define $v(b)=-1$ if $b\neq 0$, and $v(b)=q-1$ if $b=0$. If $p\nmid k$, then  $$ N_{\F}(\underbrace{1,\dots,1}_{k};b)=\frac{1}{q}(q)_k,$$
	and if $p\mid k$, then
	$$ N_{\F}(\underbrace{1,\dots,1}_{k};b)=\frac{1}{q}(q)_k+(-1)^{k+k/p}\frac{v(b)}{q}k!\binom{q/p}{k/p}.$$
\end{theorem}
Along this way, Li and Wan gave a series of asymptotic estimates
on $N_{D}({1, 1, \ldots, 1};b)$ for many different kinds of $D\subseteq \f$; see for example \cite{li2013, li2015, li2018}.

In this paper, we first prove that apart from some exceptions that can be classified the linear equation $a_1x_1+a_2x_2+\cdots+a_kx_k=1$ always has a solution with distinct coordinates.
\begin{theorem}\label{thm:existence}
	Suppose $q\geq 3$ and $k\leq q$. Then the linear equation $a_1x_1+a_2x_2+\cdots+a_kx_k=1$ has a solution $(x_1,x_2,\ldots,x_k)\in \F^k$ with all $x_i$ distinct, unless one of the following holds: {\rm (i)} $k<q$ and $a_1=a_2=\cdots=a_k=0${\rm ;} {\rm (ii)} $k=q$ and $a_1=a_2=\cdots=a_k=a$ for an element $a\in\F$.
\end{theorem}

We remark that this theorem together with Theorem 1.2 in \cite{gacs} allows us to characterize when a linear equation $a_1x_1+a_2x_2+\cdots+a_kx_k=b$ has a solution with distinct coordinates, i.e., when $N_{\F}(a_1,a_2, \ldots,a_k;b)>0$.

Next we obtain a recurrence formula for $N_{\F}(a_1,a_2, \ldots,a_k;b)$. 

\begin{theorem} \label{thm:recurrence}
If $\sum_{i=1}^k a_i\neq 0$, then
	$$ N_{\F}(a_1,a_2, \ldots,a_k;b)=\frac{1}{q}(q)_k.$$
	If $\sum_{i=1}^k a_i=0$, then
	$$N_{\F}(a_1,a_2, \ldots,a_k;b)=\frac{1}{q}(q)_k-\frac{v(b)}{q}\left( d(a_1,a_2, \ldots,a_k)+\sum_{i=1}^k d(a_1, a_2, \ldots,a_{i-1},\widehat{a_i},a_{i+1},\dots,a_k)\right),$$
	where $v(b)=-1$ if $b\neq 0$, and $v(b)=q-1$ if $b=0$;  the hat denotes the omission of an element, and $d(a_1,a_2, \ldots,a_k)$ satisfies
	\begin{equation*}
	d(a_1,a_2, \ldots,a_k)=\begin{cases}
	-v(a_1), & \text{if}\ \ k=1;\\
	-\sum_{i=1}^k d(a_1, a_2, \ldots,a_{i-1},\widehat{a_i},a_{i+1},\dots,a_k),& \text{if}\ k\geq 2\ \text{and}\ \sum_{i=1}^k a_i\neq 0;\\
	(q-k)d(a_1, a_2, \ldots,a_{i-1},\widehat{a_i},a_{i+1},\dots,a_k),&  \text{if}\ k\geq 2\ \text{and}\ \sum_{i=1}^k a_i=0.
	\end{cases}
	\end{equation*}
\end{theorem}

This immediately gives  an $O(k!)$ algorithm for computing $N_{\F}(a_1,a_2, \ldots,a_k;b)$ explicitly.
\begin{corollary}
The number $N_{\F}(a_1,a_2, \ldots,a_k;b)$ can be computed in $O(k!)$ field operations.
\end{corollary}

We also obtain an explicit formula for
 $N_{\F}(a_1,a_2, \ldots,a_k;b)$ involving combinatorial numbers depending on $a_i$'s.  In particular, we obtain a closed formula when $a_i, 1\leq i\leq k$ take at most three distinct values. This generalizes the main result of Theorem \ref{thm:li2008}.

\begin{theorem} \label{thm:summation}
	Let $p(a_1,a_2, \ldots,a_k;k,i)$ be the number of permutations in $S_k$ of $i$ cycles with the sum of $a_i$'s over its each cycle vanishing. Then
	$$N_{\F}(a_1,a_2, \ldots,a_k;b)=\frac{1}{q}(q)_k + \frac{v(b)}{q}\sum_{i=1}^n(-1)^{k-i}p(a_1,a_2, \ldots,a_k;k,i)q^i.$$
 In particular, if $\sum_{i=1}^k a_i=0$ but $\sum_{i\in I}a_i\neq 0$ for all $I\subsetneq \{1, 2, \ldots,k\}$, then
	$$ N_{\f}(a_1,a_2, \ldots,a_k;b)=\frac{1}{q}(q)_k+v(b)(-1)^{k-1}(k-1)!.$$
\end{theorem}

\begin{theorem}\label{thm:special case}
	Denote by $\{x\}_p=x-\lfloor x/p\rfloor p$ the least non-negative residue of $x$ modulo $p$. For a special case $[a_1, a_2,a_3, \dots, a_k]=[a_1, a_2, 1,\dots, 1]$, we have
	 \begin{enumerate}
	 	\item If $a_1+a_2+k-2\neq 0$, then
	 	$$ N_{\F}(a_1,a_2,\underbrace{1,\dots,1}_{k-2};b)=\frac{1}{q}(q)_k.$$
	 	\item If $a_1+a_2+k-2=0$ and $a_{1},a_{2}\notin\mathbb{F}_p$, then
	 	\begin{equation*}
	 	N_{\F}(a_1,a_2,\underbrace{1,\dots,1}_{k-2};b)=\frac{1}{q}(q)_k+v(b)(-1)^{k-1+\lfloor\frac{k-2}{p}\rfloor}(k-2)!\left(\frac{q \{ k-2 \}_p -p(k-2)}{q-p}+1\right)\binom{q/p-1}{\lfloor (k-2)/p\rfloor}.
	 	\end{equation*}
	 	\item If $a_1+a_2+k-2=0$ and $a_1,a_2\in \mathbb{F}_p$, then
	 	\begin{equation*}
	 	N_{\F}(a_1,a_2,\underbrace{1,\dots,1}_{k-2};b)=\frac{1}{q}(q)_k+v(b)(-1)^{k-1+\lfloor\frac{k-1}{p}\rfloor}(k-2)!(k-1-q1_A(a_1,a_2))\binom{q/p-1}{\lfloor (k-1)/p\rfloor},
	 	\end{equation*}
	 	where $A=\{(a_1,a_2)\in\mathbb{F}_p^2: a_1\neq 1,\ a_2\neq 1\ \text{and}\  \{ 1-a_1\}_p +\{ 1-a_2\}_p\leq p\}$.
	 \end{enumerate}
\end{theorem}

  When the field is prime and the $a_i$'s satisfy some strong conditions, this problem was first considered by Sch\"{o}nemann \cite{schonemann} 180 years ago, and he proved the following result:

\begin{theorem}[Sch\"{o}nemann]
	Let $p$ be a prime, $a_1, a_2, \ldots,a_k$ be arbitrary integers. If  $\sum_{i=1}^k a_i\equiv 0\pmod{p}$ but $\sum_{i\in I}a_i\not\equiv 0\pmod{p}$ for all $I\subsetneq \{1,\dots,k\}$,
then
	$$ N_{\ff}(a_1, a_2, \dots, a_k; 0)=\frac{1}{p}(p)_k + (-1)^{k-1}(k-1)!(p-1).$$
\end{theorem}
One may generalize this problem from $\ff$ to $\z/n\z$, i.e., the residue ring modulo $n$. Similarly, given $a_1, a_2, \ldots,a_k,b\in \z/n\z$, we define $$N_{\z/n\z}(a_1,a_2, \ldots,a_k;b)=\#\{(x_1,x_2,\ldots,x_k)\in (\z/n\z)^k:a_1x_1+a_2x_2+\cdots+a_kx_k=b:x_i\neq x_j,\ \forall\ i\neq j\}.$$
  By using tools from additive combinatorics and group theory, Grynkiweicz et al. \cite{grynkiewicz} gave necessary and sufficient conditions to characterize when  $N_{\z/n\z}(a_1,a_2, \ldots,a_k;b)>0$; see also \cite{adams,grynkiewicz} for connections to zero-sum theory and \cite{bibakcongruences} for applications to coding theory.

  Bibak et al. generalize Sch\"{o}nemann's theorem from $\ff$ to $\z/n\z$ (\cite{bibakgraph}). They proved the following result:
  \begin{theorem}[\cite{bibakgraph}, Theorem 2.3]\label{thm:bibak} %
	Let $a_1, a_2, \ldots,a_k,b,n\in\Z$, $n\geq 1$, and $\gcd(\sum_{i\in I}a_i,n)=1$ for all $I\subsetneq \{1,\dots,k\}$. The number $N_{\Z/n\Z}(a_1,a_2, \ldots,a_k;b)$ of solutions $(x_1,x_2,\ldots,x_k)\in (\Z/n\Z)^k$ of linear congruence $a_1x_1+a_2x_2+\dots+a_kx_k\equiv b \pmod{n}$ with all $x_i$ distinct modulo $n$, is
	$$ N_{\Z/n\Z}(a_1,a_2, \ldots,a_k;b)=\begin{cases}
	\frac{1}{n}(n)_k + (-1)^k(k-1)!, & \text{if}\ \gcd(\sum_{i=1}^k a_i,n)\nmid b;\\
	\frac{1}{n}(n)_k+(-1)^{k-1}(k-1)!\big(\gcd(\sum_{i=1}^k a_i,n)-1\big), &\text{if}\ \gcd(\sum_{i=1}^k a_i,n)\mid b.
	\end{cases}$$
\end{theorem}

 The main technique for counting $N_{\F}(a_1,a_2, \ldots,a_k;b)$   is a sieve method for distinct coordinate counting developed by Li and Wan in \cite{li2010} and it works well for the $\z/n\z$ case and thus we give another proof of Theorem \ref{thm:bibak}.

   This paper is organized as follows.  Some preliminary results
    and the proofs of Theorem \ref{thm:existence} and Theorem \ref{thm:recurrence} are given in Section 2. The Li-Wan sieve technique and the proof of  Theorem \ref{thm:summation} are introduced in Section 3.
     The proof for Theorem \ref{thm:special case} is given in Section 4
     and the proof for Theorem \ref{thm:bibak} is given in Section 5.

   {\bf Notations.}  We use $(q)_k:=q(q-1)\dots (q-k+1)$ to denote the falling factorial of $q$ and $\lfloor x\rfloor$ to denote the greatest integer less than or equal to $x$. If $A$ is a set, we use $1_{A}(x)$ to denote the indicator function, thus $1_A(x)=1$ when $x\in A$ and $1_A(x)=0$ otherwise.

  \section{Preliminary Results and the Proof of Theorem  \ref{thm:recurrence} }

The number of ordered $k$-tuples $(x_1,x_2,\ldots,x_k)\in\F^k$ with all $x_i$ distinct is $(q)_k$, and the sum $\sum_{i=1}^k a_ix_i$ could be any element $b$ of the finite field $\F$. One expects that in favorable cases that the sums are equally distributed and thus $N_{\F}(a_1,a_2, \ldots,a_k;b)$ should be roughly $\frac{1}{q}(q)_k$. It is indeed the case when the $a_i$'s do not sum to zero. A simple observation gives the following result.

\begin{lemma}\label{lem:NFq}
	If $\sum_{i=1}^k a_i\neq 0$, then  $N_{\F}(a_1,a_2, \ldots,a_k;b)$ are equal  for all $b\in \F$. If $\sum_{i=1}^k a_i=0$, then $N_{\F}(a_1,a_2, \ldots,a_k)$ are equal for all $b\in\F$ expect $b=0$.
\end{lemma}
\begin{proof}
	Pick an element $c\in \F$. Note that the bijective map $(x_1,x_2,\ldots,x_k)\to (x_1+c, x_2+c, \dots,x_k+c)$ sends the distinct coordinate solutions of the linear equation $a_1x_1+a_2x_2+\cdots+a_kx_k=b$ to those of the linear equation $a_1x_1+a_2x_2+\cdots+a_kx_k=b+Ac$, where $A=\sum_{i=1}^k a_i$. Thus
	$$N_{\F}(a_1,a_2, \ldots,a_k;b)=N_{\F}(a_1,a_2, \ldots,a_k;b+Ac)$$
	for any $c\in \F$. Then observe that, for $A\neq 0$, $b+Ac$ runs over all elements of $\F$ when $c$ does. Thus $N_{\F}(a_1,a_2, \ldots,a_k;b)$ are equal for all $b\in\F$ if $\sum_{i=1}^k a_i\neq 0$.
	
	For the case $\sum_{i=1}^k a_i=0$, consider the bijective map $(x_1,x_2,\ldots,x_k)\mapsto (x_1/b, x_2/b, \dots,x_k/b)$, where $b\neq 0$. It sends the distinct coordinate solution of the linear equation $a_1x_1+a_2x_2+\cdots+a_1x_k=b$ to those of the linear equation $a_1x_1+a_2x_2+\cdots+a_kx_k=1$. Thus
	$$N_{\F}(a_1,a_2, \ldots,a_k;b)=N_{\F}(a_1,a_2, \ldots,a_k;1)$$
	for all $b\neq 0$. Therefore $N_{\F}(a_1,a_2, \ldots,a_k;b)$ are equal for all $b\in \F$ except $b=0$.
\end{proof}

As an immediate consequence of Lemma \ref{lem:NFq}, we have
\begin{corollary}\label{cor:NFq}
	If $\sum_{i=1}^k a_i\neq 0$, then
	$$ N_{\F}(a_1,a_2, \ldots,a_k;b)=\frac{1}{q}(q)_k.$$
	If $\sum_{i=1}^k a_i=0$, then
	$$N_{\F}(a_1,a_2, \ldots,a_k;b)=\frac{1}{q}(q)_k-\frac{v(b)}{q}(N_{\F}(a_1,a_2, \ldots,a_k;1)-N_{\F}(a_1,a_2, \ldots,a_k;0)),$$
	where $v(b)=-1$ if $b\neq 0$, and $v(b)=q-1$ if $b=0$.
\end{corollary}
\begin{proof}
	Note that
	$$ \sum_{b\in \F}N_{\F}(a_1,a_2, \ldots,a_k;b)=(q)_k.$$
	The claim then follows from this equality and Lemma \ref{lem:NFq}.
\end{proof}

Now we  prove Theorem \ref{thm:existence}.
\begin{proof}[Proof the Theorem {\rm \ref{thm:existence}}]
If $k<q$, extend the set of $a_i$'s to a set of size $q$ with $a_{k+1}=a_{k+2}=\cdots=a_q=0$. Then notice that $N_{\F}(a_1,a_2, \ldots,a_k;1)=0$ if and only if $N_{\F}(a_1,a_2, \ldots,a_k,a_{k+1},a_{k+2},\dots,a_q;1)=0$, where $a_{k+1}=a_{k+2}=\cdots=a_q=0$, so we only need to consider the case $k=q$.  Now assume $k=q$. We shall show that $N_{\F}(a_1, a_2, \ldots, a_q;1)=0$ if and only if $a_i$, $1\leq i\leq q$ are equal. Suppose that the linear equation $a_1x_1+a_2x_2+\cdots+a_qx_q=1$ does not have a solution $(x_1,x_2, \ldots,,x_q)\in\F^q$ with all $x_i$ distinct, and thus neither does the linear equation $a_1x_1+a_2x_2+\cdots+a_qx_q=b$ with $b\neq 0$ by Lemma \ref{lem:NFq}. This implies that $a_1x_1+a_2x_2+\cdots+a_qx_q=0$ for all ordered $q$-tuples $(x_1, x_2, \ldots, x_q)\in\F^q$ with $x_i$ distinct. Let $(x_1, x_2, \ldots, x_q)\in \F^q$ be an ordered $q$-tuple with all $x_i$ distinct (there exists such an ordered $q$-tuple since $|\F|=q$), we then have
\begin{equation}\label{eq:2a}
a_1x_1 + \cdots +a_ix_i +\cdots +a_jx_j +\cdots +a_kx_k =0
\end{equation}
Swapping the $i$-th and the $j$-th coordinates of $(x_1,x_2,\ldots,x_q)$, we obtain another ordered $q$-tuple with distinct coordinates and thus
\begin{equation}\label{eq:2b}
a_1x_1 + \cdots +a_ix_j +\cdots +a_jx_i +\cdots +a_kx_k =0
\end{equation}
Subtracting \eqref{eq:2b} from \eqref{eq:2a}, we get $(a_i-a_j)(x_i-x_j)=0$, which implies $a_i=a_j$ since $x_i\neq x_j$. Since $i,j$ are arbitrary, we conclude that all of the $a_i$ are equal if there does not exist distinct $x_i\in\F$ such that $a_1x_1+a_2x_2+\cdots+a_qx_q=1$. On the other hand, if $q\geq 3$ and all of the $a_i$ are equal, then  $a_1x_1+a_2x_2+\cdots+a_qx_q=0$ for all $(x_1,x_2, \ldots,,x_q)\in\F^q$ with $x_i$ distinct since the sum of all elements of $\F$ is zero except $\F$ being $\mathbb{F}_2$. Thus there does not exist distinct $x_i\in \F$ satisfying $a_1x_1+a_2x_2+\cdots+a_qx_q=1$ when $q\geq 3$ and all of the $a_i$ are equal. The proof is completed.
\end{proof}

Next we turn to the proof of Theorem \ref{thm:recurrence}, the recurrence relation of $N_{\F}(a_1,a_2, \ldots,a_k;b)$. The main idea is to introduce $N_{\F^*}(a_1,a_2, \ldots,a_k;b)$, the number of solutions $(x_1,x_2,\ldots,x_k)\in (\F^*)^k$ of the linear equation $a_1x_1+a_2x_2+\cdots+a_kx_k=b$ with all $x_i$ distinct, which is  related to $N_{\F}(a_1,a_2, \ldots,a_k;b)$ by the Lemma given below.

\begin{lemma}\label{lem:relation1}
	Let $A=\sum_{i=1}^k a_i$ and  $c\in\F$. Then we have
	\begin{equation}\label{eq:relation1}
	N_{\F}(a_1,a_2, \ldots,a_k;b)=N_{\F^*}(a_1,a_2, \ldots,a_k;b-Ac)+\sum_{i=1}^k N_{\F^*}(a_1, a_2, \ldots,a_{i-1},\widehat{a_i},a_{i+1},\dots,a_k;b-Ac),
	\end{equation}
	where the hat denotes the omission of an element. In particular, letting $(b,c)$ be $(1,0)$ and $(0,0)$, we obtain
	\begin{align}
	N_{\F}(a_1,a_2, \ldots,a_k;1)&=N_{\F^*}(a_1,a_2, \ldots,a_k;1)+\sum_{i=1}^k N_{\F^*}(a_1, a_2, \ldots,a_{i-1},\widehat{a_i},a_{i+1},\dots,a_k;1),\label{eq:relationa}\\
	N_{\F}(a_1,a_2, \ldots,a_k;0)&=N_{\F^*}(a_1,a_2, \ldots,a_k;0)+\sum_{i=1}^k N_{\F^*}(a_1,a_2, \ldots,a_{i-1},\widehat{a_i},a_{i+1},\dots,a_k;0).\label{eq:relationb}
	\end{align}
\end{lemma}
\begin{proof}
	Let $c$ be an element of $\F$. Then the solutions $(x_1,x_2,\ldots,x_k)\in\F^k$ of the linear equation $a_1x_1+a_2x_2+\cdots+a_kx_k=b$ with all $x_i$ distinct can be divided into two parts depending on whether $c$ appears. By the linear substitution $y_i=x_i-c$, $1\leq i\leq k$, the number of solutions in which $c$ does not appear is $N_{\F^*}(a_1,a_2, \ldots,a_k;b-Ac)$, and the number of solutions in which $c$ appears is $\sum_{i=1}^k N_{\F^*}(a_1, a_2, \ldots,a_{i-1},\widehat{a_i},a_{i+1},\dots,a_k;b-Ac)$. Thus \eqref{eq:relation1} follows.
\end{proof}

There is an additional relation between $N_{\F}(a_1,a_2, \ldots,a_k;b)$ and $N_{\F^*}(a_1,a_2, \ldots,a_k;b)$ when $\sum_{i=1}^k a_i=0$.
\begin{lemma}\label{lem:relation2}
	Suppose that $\sum_{i=1}^k a_i=0$. Then we have
	\begin{equation*}
	N_{\F}(a_1,a_2, \ldots,a_k;b)=qN_{\F^*}(a_1, a_2, \ldots,a_{i-1},\widehat{a_i},a_{i+1},\dots,a_k;b)
	\end{equation*}
	for  $1\leq i\leq k$ and $b\in\F$. In particular, $N_{\F^*}(a_1, a_2, \ldots,a_{i-1},\widehat{a_i},a_{i+1}\dots,a_k;b)$ are equal for all $1\leq i\leq k$.
\end{lemma}

\begin{proof}
   By the linear substitution $y_1=x_1$ and $y_i=x_i-x_1$, $2\leq i\leq k$ and the assumption that $\sum_{i}^k a_i=0$, the number of solutions $(x_1,x_2,\ldots,x_k)\in \F^k$ of $a_1x_1+a_2x_2+\cdots+a_kx_k=b$ with all $x_i$ distinct is equal to the number of solutions $(y_1,y_2\dots,y_k)\in \F\times (\F^*)^{k-1}$ of  $a_2y_2+ \dots+a_ny_n=b$  with $y_i$ distinct for $2\leq i\leq k$. Since $y_1\in \F$ can be arbitrarily chosen, we have
	\begin{equation*}
	N_{\F}(a_1,a_2, \ldots,a_k;b)=qN_{\F^*}(a_2,\dots,a_k;b)
	\end{equation*}
	The same argument gives
	\begin{equation*}
N_{\F}(a_1,a_2, \ldots,a_k;b)=qN_{\F^*}(a_1, a_2, \ldots,a_{i-1},\widehat{a_i},a_{i+1},\dots,a_k;b)
	\end{equation*}
	for all $1\leq i\leq k$. Hence the proof is completed.
\end{proof}

\begin{proof}[Proof of Theorem {\rm \ref{thm:recurrence}}]
 For an ordered $k$-tuple $(a_1,a_2, \ldots,a_k)\in\F^k$, define
\begin{equation}\label{eq:difference}
d(a_1,a_2, \ldots,a_k):=N_{\F^*}(a_1,a_2, \ldots,a_k;1)-N_{\F^*}(a_1,a_2, \ldots,a_k;0).
\end{equation}
Subtracting \eqref{eq:relationb} from \eqref{eq:relationa}, we obtain
\begin{equation}\label{eq:relationc}
N_{\F}(a_1,a_2, \ldots,a_k;1)-N_{\F}(a_1,a_2, \ldots,a_k;0)=d(a_1,a_2, \ldots,a_k)+\sum_{i=1}^k d(a_1, a_2, \ldots,a_{i-1},\widehat{a_i},a_{i+1},\dots,a_k).
\end{equation}
  It is direct to verify  that $d(a_1)=N_{\F^*}(a_1;1)-N_{\F^*}(a_1;0)$ equals $-(q-1)$ if $a_1=0$, and equals $1$ if $a_1\neq 0$. Thus $d(a_1)=-v(a_1)$.

	Now suppose $k\geq 2$. If $\sum_{i=1}^k a_i\neq 0$, then by Lemma \ref{lem:NFq} we have $N_{\F}(a_1,a_2, \ldots,a_k;1)=N_{\F}(a_1,a_2, \ldots,a_k;0)$. Thus the left-hand side of \eqref{eq:relationc} is zero, which implies
	\begin{equation*}
	d(a_1,a_2, \ldots,a_k)=-\sum_{i=1}^k d(a_1,a_2, \ldots, a_{i-1},\widehat{a_i},a_{i+1},\dots,a_k).
	\end{equation*}
	Now let us consider the case $\sum_{i=1}^k a_i=0$. In this case, Lemma \ref{lem:relation2} implies that the left-hand side of \eqref{eq:relationc} is equal to $qd(a_1,a_2, \ldots, a_{i-1},\widehat{a_i},a_{i+1},\dots,a_k)$.
	Again, by Lemma \ref{lem:relation2},  $d(a_1, a_2, \ldots,a_{i-1},\widehat{a_i},a_{i+1},\dots,a_k)$ are equal for $1\leq i\leq k$ when  $\sum_{i=1}^k a_i=0$. Thus the right-hand side \eqref{eq:relationc} can be simplified into
	$d(a_1,a_2, \ldots,a_k)+kd(a_1, a_2, \ldots,a_{i-1},\widehat{a_i},a_{i+1},\dots,a_k)$. Therefore equality \eqref{eq:relationc} yields
	\begin{equation*}
	d(a_1,a_2, \ldots,a_k)=(q-k)d(a_1, a_2, \ldots,a_{i-1},\widehat{a_i},a_{i+1},\dots,a_k).
	\end{equation*}
\end{proof}

\section{Li-Wan's new sieve and the summation expression}

In \cite{li2010}, J.  Li and D. Wan proposed a new sieve method for  distinct coordinate counting problems. We introduce it here briefly.

Let $D$ be a finite set. For a positive integer $k$, let $D^k=D\times D\times \cdots\times D$ be the $k$-fold Cartesian product of $D$ with itself. Let $X$ be a subset of $D^k$. Then every element $x\in X$ can written in an ordered $k$-tuple form $x=(x_1,x_2,\ldots,x_k)$ with $x_i\in D$, $1\leq i\leq k$. We are interested in the number of the elements in $X$ whose coordinates are distinct, that is, the cardinality of the set
\begin{equation}\label{def:overlineX}
\overline{X}=\{(x_1,x_2,\ldots,x_k)\in X:x_i\neq x_j,\ \forall\, i\neq j\}.
\end{equation}
Let $S_k$ denote the symmetric group on the set $\{1,\dots,k\}$. Given a permutation $\tau\in S_k$, we can write it as a product of disjoint cycles $\tau=C_1C_2\cdots C_{\ell(\tau)}$ uniquely apart from the order of the cycles, where $\ell(\tau)$ denotes the number of disjoint cycles of $\tau$. We define the {\it signature} of $\tau$ to be $\sign(\tau):=(-1)^{k-\ell(\tau)}$. We also define the set $X_\tau$ to be
\begin{equation}\label{def:Xtau}
X_\tau:=\{(x_1,x_2,\ldots,x_k)\in X:x_i\ \text{are equal for}\ i\in C_j,\ 1\leq j\leq \ell(\tau)\}
\end{equation}
We have the following theorem which will be used in the proof of the summation expression of $N_{\F}(a_1,a_2, \ldots,a_k;b)$.
\begin{theorem}[ \cite{li2010}, Theorem 1.1]\label{thm:sieve} %
	We have
	$$ |\overline{X}|=\sum_{\tau\in S_k}\sign(\tau)|X_\tau|.$$
\end{theorem}

The unsigned Stirling number of the first kind $c(k,i)$ is defined to be the number of permutations in $S_k$ with exactly $i$ cycles. It can also be defined via the following classic identity \cite{stanley}:
\begin{equation}\label{eq:identity}
\sum_{i=0}^k (-1)^{k-i}c(k,i)q^i=(q)_k.
\end{equation}

With these preparations, we are ready to prove Theorem \ref{thm:summation}.
\begin{proof}[Proof of Theorem {\rm \ref{thm:summation}}]
	Let $X$ be the set of all solutions (not necessarily to have distinct coordinates) of the linear equation $a_1x_1+a_2x_2+\cdots+a_kx_k=b$ in $\F$, i.e.,
	$$ X=\{(x_1,x_2,\ldots,x_k)\in\F^k:a_1x_1+a_2x_2+\cdots+a_kx_k=b\}.$$
 	Since $N_{\F}(a_1,a_2, \ldots,a_k;b)$ counts the number of elements in $\overline{X}$, by Theorem \ref{thm:sieve} we have
	\begin{equation}\label{eq:NFq}
	N_{\F}(a_1,a_2, \ldots,a_k;b)=|\overline{X}|=\sum_{\tau\in S_k}(-1)^{\sign(\tau)}|X_\tau|,
	\end{equation} 	
	where  $\overline{X}$ and $X_\tau$ are defined as in \eqref{def:overlineX} and \eqref{def:Xtau}. Let $\tau=C_1C_2\cdots C_\ell$ be the disjoint cycle product of $\tau$. Let $A_j=\sum_{i\in C_j}a_i$, $1\leq j\leq \ell$. By the definition of $X_\tau$, we have
	$$X_\tau =\{(y_1,y_2, \ldots, y_\ell)\in \F^\ell:A_1y_1+A_2y_2+\cdots+A_\ell y_\ell=b\}.$$
	Hence  $|X_\tau|=q^{\ell-1}=q^{\ell(\tau)-1}$ if $A_1, A_2, \ldots,A_\ell$ are not all zero, and $|X_\tau|= q^\ell1_{b=0}=q^{\ell(\tau)-1}(v(b)+1)$ otherwise.
      Denote by $p(a_1,a_2, \ldots,a_k;k,i)$ the number of permutations in $S_k$ of $i$ cycles with the sum of $a_i$'s over its each cycle vanishing. We deduce from \eqref{eq:NFq} that
      	\begin{align*}
	N_{\F}(a_1,a_2, \ldots,a_k;b)&=\sum_{i=1}^k \sum_{\tau\in S_k:\ell(\tau)=i}(-1)^{k-i}|X_\tau|\\	
	&=\sum_{i=1}^k (-1)^{k-i}(c(k,i)-p(a_1,a_2, \ldots,a_k;k,i))q^{i-1}\\
	&\quad+\sum_{i=1}^k (-1)^{k-i}p(a_1,a_2, \ldots,a_k;k,i)q^{i-1}(v(b)+1)\\
	&=\frac{1}{q}\sum_{i=1}^k (-1)^{k-i}c(k,i)q^{i}+v(b)\sum_{i=1}^k (-1)^{k-i} p(a_1,a_2, \ldots,a_k;k,i)q^{i-1}\\
	&=\frac{1}{q}(q)_k+\frac{v(b)}{q}\sum_{i=1}^k (-1)^{k-i} p(a_1,a_2, \ldots,a_k;k,i)q^i.
	\end{align*}
	
In particular, if $\sum_{i=1}^k a_i=0$ but $\sum_{i\in I}a_i\neq 0$ for all $I\subsetneq \{1,\dots,k\}$, then $p(a_1,a_2, \ldots,a_k,k,i)=0$ for $i\geq 2$ and $p(a_1,a_2, \ldots,a_k,k,i)=(k-1)!$ for $i=1$. Thus  we conclude that
$$	N_{\F}(a_1,a_2, \ldots,a_k;b)=\frac{1}{q}(q)_k +v(b)(-1)^{k-1}(k-1)!$$
in this special case.
\end{proof}

\section{Proof of Theorem \ref{thm:special case}}

In this section, we prove Theorem \ref{thm:special case}. We first need some combinatorial formulas and equalities.

\begin{lemma}\label{lem:binomial}
	Let $k,n$ be integers. Then we have
	\begin{equation}\label{eq:binomial1}
	\sum_{j=0}^k (-1)^j\binom{n}{j} = (-1)^k\binom{n-1}{k}
	\end{equation}
	and
	\begin{equation}\label{eq:binomial2}
	\sum_{j=0}^k (-1)^j j\binom{n}{j} =(-1)^k n\binom{n-2}{k-1}.
	\end{equation}
\end{lemma}

\begin{proof}
	Comparing the coefficients $x^k$ on both sides of the identity $(1-x)^{-1}(1-x)^n=(1-x)^{n-1}$, we obtain \eqref{eq:binomial1}. Similarly, comparing the coefficients of $x^{k-1}$ on both sides of the identity $(1-x)^{-1}((1-x)^n)'=-n(1-x)^{n-2}$, we obtain \eqref{eq:binomial2}.
\end{proof}

\begin{lemma}[\cite{li2010}, Lemma 3.1]\label{lem:p(k,i)}
	Assume $p\mid k$. Let $p(k,i)$ be the number of permutations in $S_k$ of $i$ cycles with the length of its each cycle divisible by $p$. Then
	$$\sum_{i=1}^k (-1)^i p(k,i)q^i =(-1)^{k/p}k!\binom{q/p}{k	/p}.$$
\end{lemma}

\begin{lemma}\label{lem:p(k,i,j)}
	Assume $p\mid (k-j)$. Let $p(k,i,j)$ be the number of permutations in $S_k$ of $i$ cycles with a cycle of length $j$ containing $\{1,2\}$ and the length of each remaining $(i-1)$ cycles divisible by $p$. Then
	$$ p(k,i,j)=(j-1)\frac{(k-2)!}{(k-j)!}p(k-j,i-1).$$
\end{lemma}
\begin{proof}
	Let $\tau\in S_k$ be a cycle described in the Lemma. We can write $\tau$ as a product of two permutations $\tau=\tau_1\tau_2$, where $\tau_1$ denotes the cycle of $\tau$ of length $j$  containing  $\{1,2\}$, and $\tau_2$ denotes the product of the other $(i-1)$ cycles of $\tau$.
	
	Since the $j$-cycle $\tau_1$ contains $\{1,2\}$ already, the remaining $(j-2)$ elements of $\tau_1$ must come from the set $\{3,\dots,k\}$ and thus there are $\binom{k-2}{j-2}$ choices of them. The number of $j$-cycles on a $j$-element set is $(j-1)!$, so there are $(j-1)!\binom{k-2}{j-2}$ ways to determine $\tau_1$ by the multiplication principle. But $\tau_2$ can be viewed as a permutation in $S_{k-j}$ of $(i-1)$ cycles such that the length of each its cycle is divisible by $p$, so there are $p(k-j,i-1)$ choices of $\tau_2$ by Lemma \ref{lem:p(k,i)}.  Since every permutation can be expressed by a product of disjoint cycles uniquely up to the order of the cycles, we see that every ordered pair $(\tau_1,\tau_2)$ uniquely corresponds to a  $\tau$. Therefore there are	 $$(j-1)!\binom{k-2}{j-2}p(k-j,i-1)=(j-1)\frac{(k-2)!}{(k-j)!}p(k-j,i-1)$$ such $\tau$'s in total. The claim then follows.
\end{proof}

\begin{lemma}\label{lem:p(k,i,j1,j2)}
		Assume $p\mid (k-j_1-j_2)$. Let $p(k,i,j_1,j_2)$ be the number of permutations in $S_k$ of $i$ cycles with a cycle  of length $j_1$  containing $\{1\}$ but not containing $\{2\}$, a cycle of length $j_2$  containing $\{2\}$ but not containing  $\{1\}$,  and the length of each remaining  $(i-2)$ cycles divisible by $p$. Then we have
	$$ p(k,i,j_1,j_2)=\frac{(k-2)!}{(k-j_1-j_2)!}p(k-j_1-j_2,i-2).$$
\end{lemma}
\begin{proof}
	Let $\tau$ be a cycle described in the Lemma. We can write $\tau$ as a product of three permutations $\tau=\tau_1\tau_2\tau_3$, where $\tau_1$ denotes the cycle of $\tau$ of length $j_1$ containing $\{1\}$ but  not containing $\{2\}$,  $\tau_2$ denotes the cycle of $\tau$ of length $j_2$ containing $\{2\}$ but not containing $\{1\}$, and $\tau_3$ denotes the product of the other $(i-2)$ cycles of $\tau$. By a similar argument used in the proof of Lemma \ref{lem:p(k,i,j)},  we conclude that there are
	$$(j_1-1)!(j_2-1)!\binom{k-2}{j_1+j_2-2}\binom{j_1+j_2-2}{j_1-1,j_2-1}p(k-j_1-j_2,i-2)=\frac{(k-2)!}{(k-j_1-j_2)!}p(k-j_1-j_2,i-2)$$
such $\tau$'s in total. The claim then follows  	
\end{proof}

\begin{proof}[Proof of Theorem {\rm \ref{thm:special case}}]
	By Corollary \ref{cor:NFq}, it suffices to consider the case $a_1+a_2+k-2=0$. So assume $a_1+a_2+k-2=0$. In particular,  both $a_1$ and $a_2$ lie in $\mathbb{F}_p$, or neither $a_1$ nor $a_2$ lies in $\mathbb{F}_p$. Let $X$ denote the set of all solutions of the linear equation $a_1x_1+a_2x_2+x_3+\cdots+x_k=b$ in $\F$, i.e.,
	$$ X=\{(x_1,x_2,\ldots,x_k)\in \F^k:a_1x_1+a_2x_2+x_3+\cdots+x_k=b\}.$$
	Then $N_{\F}(a_1,a_2,\underbrace{1,\dots,1}_{k-2};b)$ counts the number of elements in $\overline{X}$.  We have from Theorem \ref{thm:sieve} that
	\begin{equation}\label{eq:NFq4} N_{\F}(a_1,a_2,\underbrace{1,\dots,1}_{k-2};b)=|\overline{X}|=\sum_{\tau\in S_k} \sign(\tau)|X_\tau|.
	\end{equation}
 For a permutation $\tau\in S_k$, write it as a disjoint cycle product $\tau=C_1C_2\cdots C_\ell$. Then we have two cases: $\{1\}$ and $\{2\}$ are contained in one cycle of $\tau$, or they are contained in two separate cycles of $\tau$, respectively. For the former case,  we may assume that $\{1,2\}$ is contained in the cycle $C_1$ after rearranging the cycles. Then we have
$$X_\tau =\{(y_1, y_2, \ldots,y_\ell)\in \F^\ell: (a_1+a_2-2+c_1) y_1+c_2y_2+\cdots+c_\ell y_\ell=b\}.$$	
where $c_i$ denotes the length of the cycle $C_i$, $1\leq i\leq \ell$. Thus $|X_\tau|=q^{\ell}1_{b=0}=q^{\ell(\tau)-1}(v(b)+1)$ if $ c_1 \equiv 2-a_1-a_2 \equiv k\pmod{p}$ and $p\mid c_i$ for $2\leq i\leq\ell$, and $|X_{\tau}|=q^{\ell-1}$ otherwise. For the latter case, we may assume that $\{1\}$ is contained in the cycle $C_1$ and $\{2\}$ is contained in the cycle $C_2$ after rearranging the cycles. Similarly, we have
$$X_\tau =\{(y_1, y_2, \ldots,y_\ell)\in \F^\ell: (a_1-1+c_1) y_1+(a_2-1+c_2)y_2+c_3y_3+\cdots+c_\ell y_\ell=b\}.$$
Thus $|X_\tau|=q^{\ell-1}(v(b)+1)$ if $ c_1 \equiv 1-a_1\pmod{p}$, $c_2\equiv 1-a_2\pmod{p}$ and $p\mid c_i$ for $3\leq i\leq \ell$, and $|X_\tau|=q^{\ell-1}$ otherwise. From this classification of $|X_\tau|$, we see that \eqref{eq:NFq4} can be simplified into
\begin{align*}
	N_{\F}(a_1,a_2,\underbrace{1,\dots,1}_{k-2};b)&=\sum_{i=1}^k\sum_{\tau\in S_k:\ell(\tau)=i} (-1)^{k-i}|X_\tau|\\
	& =\frac{1}{q}(q)_k +v(b) (S_1+S_2),
\end{align*}
where
\begin{align*}
S_1 &=\sum_{i=1}^k (-1)^{k-i}\sum_{\substack{2\leq j\leq k\\ j\equiv k\ ({\rm mod}\ p)}}p(k,i,j) q^{i-1},\\
S_2 &=\sum_{i=1}^k (-1)^{k-i}\sum_{\substack{1\leq j_1,j_2\leq k\\ j_1+j_2\leq k\\ j_1\equiv 1-a_1\ ({\rm mod}\ p) \\ j_2\equiv 1-a_2\ ({\rm mod}\ p)}} p(k,i,j_1,j_2) q^{i-1}.
\end{align*}

We first consider the case $a_1,a_2\notin \mathbb{F}_p$ in which it suffices to evaluate the sum $S_1$ since the sum $S_2$ vanishes.
Applying Lemma \ref{lem:p(k,i,j)} and Lemma \ref{lem:p(k,i)}, we see that
\begin{align*}
S_1 & =\sum_{i=1}^k(-1)^{k-i}\sum_{\substack{2\leq j\leq k\\ j\equiv k\ ({\rm mod}\ p)}}(j-1)\frac{(k-2)!}{(k-j)!}p(k-j,i-1) q^{i-1}\\
&=(-1)^{k-1}(k-2)! \sum_{\substack{2\leq j\leq k\\ j\equiv k\ ({\rm mod}\ p)}} (j-1)(-1)^{\frac{k-j}{p}}\binom{q/p}{(k-j)/p}\\
& =(-1)^{k-1}(k-2)! \sum_{0\leq \ell\leq \lfloor\frac{k-2}{p}\rfloor} (k-1-p\ell)(-1)^{\ell}\binom{q/p}{\ell}
\end{align*}
We can evaluate the above sum by using Lemma \ref{lem:binomial} and thus obtain
\begin{align*}
S_1 &=(-1)^{k-1+\lfloor\frac{k-2}{p}\rfloor}(k-2)!\left((k-1)\binom{q/p-1}{\lfloor (k-2)/p\rfloor}-q\binom{q/p-2}{\lfloor (k-2)/p\rfloor-1}\right)\\
&=(-1)^{k-1+\lfloor\frac{k-2}{p}\rfloor}(k-2)!\left(\frac{q\{ k-2\}_p -p(k-2)}{q-p}+1\right)\binom{q/p-1}{\lfloor (k-2)/p\rfloor}.
\end{align*}

For the case $a_1,a_2\in \mathbb{F}_p$, we have to consider the sum $S_2$. Similarly, using Lemma \ref{lem:p(k,i,j1,j2)} and Lemma \ref{lem:p(k,i)}, we see that
\begin{align*}
S_2&=\sum_{i=1}^k (-1)^{k-i}\sum_{\substack{1\leq j_1,j_2\leq k\\ j_1+j_2\leq k\\ j_1\equiv 1-a_1\ ({\rm mod}\ p) \\ j_2\equiv 1-a_2\ ({\rm mod}\ p)}} \frac{(k-2)!}{(k-j_1-j_2)!}p(k-j_1-j_2,i-2) q^{i-1}\\
&=(-1)^{k-2}q(k-2)!\sum_{\substack{1\leq j_1,j_2\leq k\\ j_1+j_2\leq k\\ j_1\equiv 1-a_1\ ({\rm mod}\ p) \\ j_2\equiv 1-a_2\ ({\rm mod}\ p)}} (-1)^{\frac{k-j_1-j_2}{p}}\binom{q/p}{(k-j_1-j_2)/p}\\
&=(-1)^{k-2}q(k-2)! \sum_{\substack{2\leq j\leq k\\ j\equiv k \ ({\rm mod}\ p)}}  (-1)^{\frac{k-j}{p}}N_j\binom{q/p}{(k-j)/p},
\end{align*}
where $N_j$ is defined as
$$N_j:=\# \{(j_1,j_2):j_1+j_2=j,\ 1\leq j_1, j_2\leq k,\ j_1\equiv 1-a_1\ ({\rm mod}\ p),\ j_2\equiv 1-a_2\ ({\rm mod}\ p) \}.$$
For $j\equiv k\pmod{p}$, it is direct to check that
$$ N_j=\begin{cases}
\lfloor (j/p\rfloor -1)1_{\lfloor j/p\rfloor\geq 1}, & \text{if}\ a_1=1, a_2=1;\\
\lfloor j/p \rfloor, & \text{if}\ a_1=1, a_2\neq 1\ \text{or}\ a_1\neq 1, a_2=1\ \text{or}\ \{ 1-a_1\}_p+ \{ 1-a_2\}_p\geq p; \\
\lfloor j/p \rfloor+1, & \text{if}\ a_1\neq 1, a_2\neq 1\ \text{and}\ \{1-a_1\}_p+ \{1-a_2\}_p<p.\\
\end{cases}$$
Inserting $N_j$ into the sum of $S_2$, by a routine computation and Lemma \ref{lem:binomial}, we obtain
\begin{align*}
S_2 = (-1)^{k-2+\lfloor\frac{k-2}{p}\rfloor}q(k-2)! \left( (\lfloor k/p\rfloor+\varepsilon)\binom{q/p-1}{\lfloor (k-2)/p\rfloor}-q/p\binom{q/p-2}{\lfloor (k-2)/p\rfloor-1} \right),
\end{align*}
where $\varepsilon$ equals $-1$, $0$, $1$ in the three cases, respectively. Again,  a direct computation shows that $S_1+S_2$ takes the form 
\begin{align*}
S_1+S_2 =(-1)^{k-1+\lfloor\frac{k-1}{p}\rfloor}(k-2)!(k-1-q1_A)\binom{q/p-1}{\lfloor (k-1)/p\rfloor},
\end{align*}
where $A=\{(a_1,a_2)\in\mathbb{F}_p^2: a_1\neq 1,\ a_2\neq 1\ \text{and}\  \{ 1-a_1\}_p +\{ 1-a_2\}_p\leq p\}$. The proof is then completed
\end{proof}

\section{Proof of Theorem \ref{thm:bibak}}
For the purpose of our proof, we need the following result of D. N. Lehmer \cite{lehmer} which gives the number of solutions of linear congruence.

\begin{proposition}[\cite{lehmer}]\label{prop:lehmer}
	Let $a_1, a_2, \ldots,a_k,b,n\in\Z$, $n\geq 1$. The linear congruence $a_1x_1+a_2x_2+\cdots+a_kx_k\equiv b \pmod{n}$ has a solution $(x_1,x_2,\ldots,x_k)\in (\Z/n\Z)^k$ if and only if $d\mid b$, where $d=\gcd(a_1,a_2, \ldots,a_k,n)$. Furthermore, if this conditions is satisfied, then there are $dn^{k-1}$ solutions.
\end{proposition}

Now we give another proof of Theorem \ref{thm:bibak} via a sieve method.

\begin{proof}[Proof of Theorem {\rm \ref{thm:bibak}}]
	Let $X$ be of set of all solutions of the linear congruence $a_1x_1+a_2x_2+\cdots+a_kx_k\equiv b\pmod{n}$ in $\Z/n\Z$, i.e.,
	$$ X=\{(x_1,x_2,\ldots,x_k)\in (\Z/n\Z)^k:a_1x_1+a_2x_2+\dots+a_kx_k\equiv b\pmod{n}\}.$$
	Then $N_{\Z/n\Z}(a_1,a_2, \ldots,a_k;b)$ counts the number of elements in $\overline{X}$.	Thus Theorem \ref{thm:sieve} yields
	\begin{equation}\label{eq:NZn}
	N_{\Z/n\Z}(a_1,a_2, \ldots,a_k;b) =|\overline{X}|=\sum_{\tau\in S_k}(-1)^{\sign(\tau)}|X_\tau|.
	\end{equation}

	Next we compute $|X_\tau|$. Let $\tau=C_1C_2\cdots C_\ell$ be a disjoint cycle product of $\tau$, and let $A_j=\sum_{i\in C_j}a_i, 1\leq j\leq\ell$. From the definition of $X_\tau$, we see that
	$$X_\tau=\{(y_1, y_2, \ldots,y_\ell)\in (\Z/n\Z)^\ell:A_1y_1+A_2y_2+\cdots+A_\ell y_\ell \equiv b\pmod{n}\}.$$
	Since $\gcd(\sum_{i\in I}a_i,n)=1$ for all $I\subsetneq \{1,\dots,k\}$,  by Proposition \ref{prop:lehmer}, we have $|X_\tau|=n^{\ell-1}=n^{\ell(\tau)-1}$ for all $\tau\in S_k$ with $\ell(\tau) \geq 2$ . Note that Proposition  \ref{prop:lehmer} also shows that
	$$|X_\tau|=\begin{cases}
	0, & \text{if} \ \gcd(\sum_{i=1}^k a_i,n)\nmid b;\\
	\gcd(\sum_{i=1}^k a_i,n), & \text{if}\ \gcd(\sum_{i=1}^k a_i,n)\mid b
	\end{cases}$$ for $\tau\in S_k$ with $\ell(\tau)=1$.
	Substituting this into \eqref{eq:NZn}, when $\gcd(\sum_{i=1}^k a_i,n)\nmid b$, we obtain
	\begin{align*}
	N_{\Z/n\Z}(a_1,a_2, \ldots,a_k;b)&=\sum_{i=1}^k\sum_{\tau\in S_k:\ell(\tau)=i}(-1)^{k-i}|X_\tau|\\
	 &= \sum_{i=2}^k (-1)^{k-i} c(k,i) n^{i-1} \\
	&=\frac{1}{n}\sum_{i=1}^k (-1)^{k-i}c(k,i)n^i -(-1)^{k-1}c(k,1)\\
	& =\frac{1}{n}(n)_k + (-1)^{k}(k-1)!
	\end{align*}
	And when $\gcd(\sum_{i=1}^k a_i,n)\mid b$, we obtain
	\begin{align*}
	N_{\Z/n\Z}(a_1,a_2, \ldots,a_k;b) &= \sum_{i=2}^k (-1)^{k-i} c(k,i) n^{i-1} + (-1)^{k-1}c(k,1)\gcd(\sum_{i=1}^k a_i,n) \\
	& =\frac{1}{n}(n)_k + (-1)^{k-1}(k-1)!\Big(\gcd(\sum_{i=1}^k a_k,n)-1\Big).
	\end{align*}
	This ends the proof
\end{proof}

\section*{Acknowledgements}
Part of the work was done when the second author was a graduate student  at Shanghai Jiao Tong University.

\end{document}